\documentclass[12pt]{article}
\usepackage{graphicx} 

\title{Blocks whose defect groups are Suzuki $2$-groups}
\author{Charles W. Eaton\footnote{Department of Mathematics, University of Manchester, Manchester M13 9PL. Email: charles.eaton@manchester.ac.uk}}
\date{13th September 2024}

\usepackage{amssymb, enumerate}
\usepackage{latexsym}
\usepackage{url}
\usepackage{amsthm}

\usepackage[all,cmtip]{xy}
\usepackage{amsmath,amssymb}
\usepackage{theoremref}

\newtheorem{theorem}{Theorem}[section]
\newtheorem{lemma}[theorem]{Lemma}
\newtheorem{corollary}[theorem]{Corollary}
\newtheorem{proposition}[theorem]{Proposition}

\theoremstyle{definition}

\newtheorem{remark}[theorem]{Remark}

\newtheorem{step}{Step}

\newcommand{\Syl}{\mathop{\rm Syl}\nolimits}
\newcommand{\Irr}{\mathop{\rm Irr}\nolimits}

\newcommand{\Aut}{\mathop{\rm Aut}\nolimits}
\newcommand{\Out}{\mathop{\rm Out}\nolimits}

\newcommand{\im}{\mathop{\rm Im}\nolimits}

\newcommand{\NN} {\mathbb{N}}

\newcommand{\FF} {\mathbb{F}}
\newcommand{\cA} {\mathcal{A}}
\newcommand{\cO} {\mathcal{O}}

\newcommand{\cB} {\mathcal{B}}
\newcommand{\cC} {\mathcal{C}}
\newcommand{\cD} {\mathcal{D}}

\def\a{\alpha}
\def\b{\beta}
\def\g{\gamma}

\def\bigcp{\mathop{\mathchoice 
 {\hbox{\sf\Large\lower 0.1\baselineskip\hbox{Y}}}%
 {\hbox{\sf\large\lower 0.1\baselineskip\hbox{Y}}}%
 {\hbox{\sf\normalsize\lower 0.1\baselineskip\hbox{Y}}}%
 {\hbox{\sf\tiny\lower 0.1\baselineskip\hbox{Y}}}%
}}
\def\bigtimes{\mathop{\mathchoice 
 {\hbox{\sf\Large\lower 0.1\baselineskip\hbox{X}}}%
 {\hbox{\sf\large\lower 0.1\baselineskip\hbox{X}}}%
 {\hbox{\sf\normalsize\lower 0.1\baselineskip\hbox{X}}}%
 {\hbox{\sf\tiny\lower 0.1\baselineskip\hbox{X}}}%
}}

\def\Sym(#1){\mathop{\rm Sym}(#1)}
\def\Sym(#1){S_{#1}}
\def\diag(#1){\mathop{\rm diag}(#1)}

\newenvironment{enumerate*}{%
 \begin{enumerate}%
 }%
 {\end{enumerate}}

\begin{document}


\maketitle


\begin{abstract}
We classify up to Morita equivalence all blocks whose defect groups are Suzuki $2$-groups. The classification holds for blocks over a suitable discrete valuation ring as well as for those over an algebraically closed field, and in fact holds up to basic Morita equivalence. As a consequence Donovan's conjecture holds for Suzuki $2$-groups. A corollary of the proof is that Suzuki Sylow $2$-subgroups of finite groups with no nontrivial odd order normal subgroup are trivial intersection.

\medskip

\end{abstract}


\section{Introduction}

Let $p$ be a prime and $(K,\cO,k)$ be a modular system with $k$ an algebraically closed field of characteristic $p$. Donovan's conjecture, which may be stated over $\cO$ or $k$, predicts that for a given finite $p$-group $P$, there are only finitely many Morita equivalence classes of blocks of finite groups with defect groups isomorphic to $P$. Further, we may ask for classifications of Morita equivalence classes of blocks with a given defect group. Most progress so far has been for tame blocks and for abelian defect groups, and there are relatively few other classes of nonabelian $p$-groups for which the conjecture or a classification is known, aside from those admitting only nilpotent blocks. See~\cite{el23} for a recent summary, and also~\cite{wiki} where progress is recorded.  

Following~\cite{hi63} a Suzuki $2$-group is a non-abelian $2$-group $P$ with more than one involution for which there is $\varphi \in \Aut(P)$ permuting the involutions in $P$ transitively. It is shown in~\cite{hi63} that $\Omega_1(P)=Z(P)=\Phi(P)=[P,P]$, so that $P$ has exponent $4$, and a characterisation of these groups is given, placing them in four infinite series labelled $\cA$ to $\cD$. Suzuki $2$-groups of type $\cA$ have order $|Z(P)|^2$ and the others have order $|Z(P)|^3$. The Suzuki $2$-groups include the Sylow $2$-subgroups of the Suzuki nonabelian simple groups and of $PSU_3(2^n)$, which we note feature as the only examples of nonabelian simple groups with nonabelian trivial intersection Sylow $2$-subgroups (see~\cite{su64}).

Here we determine the Morita equivalence classes of blocks whose defect groups are Suzuki $2$-groups. An important part of this is the observation that blocks with such defect groups are controlled, so that the classification by An in~\cite{an20} of controlled $2$-blocks of quasisimple groups may be applied. 

Recall for the following that a Morita equivalence is \emph{basic} if it is induced by an endopermutation source bimodule (see~\cite{pu99}). Our main result is as follows:

\begin{theorem}
\thlabel{main_theorem}
Let $G$ be a finite group and $B$ be a block of $\cO G$ with defect group $P$ a Suzuki $2$-group. Then $B$ is basic Morita equivalent to one of the following:
\begin{enumerate}[(i)]
\item a block of $P \rtimes \hat{E}$ where $E$ is an odd order subgroup of $\Aut(P)$ and $\hat{E}$ is a central extension of $E$ by $Z$ with $Z \leq [\hat{E},\hat{E}]$ acting trivally on $P$;
\item the principal block of $H$ for $^2B_2(2^{2n+1}) \leq H \leq \Aut({}^2B_2(2^{2n+1}))$ for some $n \geq 1$;
\item a block of maximal defect of $H$ where $Z(H) \leq [H,H]$ and $PSU_3(2^n) \leq H/Z(H) \leq \Aut(PSU_3(2^n))$ for some $n \geq 2$ with $[H/Z(H):PSU_3(2^n)]$ odd.
\end{enumerate}

Further, the Morita equivalence preserves the isomorphism type of the defect group, the Frobenius category and the K\"ulshammer-Puig class (see Section \ref{Controlled}). 
\end{theorem}

As an almost immediate consequence we have:

\begin{corollary}
\thlabel{Donovan}
Donovan's conjecture holds for Suzuki $2$-groups. In fact, there are only finitely many of basic Morita equivalence classes of blocks with defect group a given Suzuki $2$-group.
\end{corollary}

The proof of \thref{main_theorem} involves a detailed analysis of the structure of groups admitting a block whose defect groups are Suzuki $2$-groups, and this analysis, together with the results of~\cite{an20}, gives us that Suzuki Sylow $2$-subgroups are always trivial intersection for groups with no nontrivial normal subgroup of odd order.

The structure of the paper is as follows. In Section \ref{Suzuki} we recall the definition of the Suzuki $2$-groups and give some properties that will be useful later. In Section \ref{Controlled} we recall controlled blocks and inertial quotients, and apply the relevant results of~\cite{an20}. Section \ref{Proofs} contains the proof of the main result. This mainly consists of the description of the structure of what we will call reduced blocks with Suzuki $2$-groups as their defect groups. We also give the full classification of blocks with Suzuki $2$-groups of order $64$ to illustrate our main result. In Section \ref{TI} we apply the description of reduced blocks from Section \ref{Proofs} to deduce the result on the trivial intersection of Suzuki Sylow $2$-subgroups. Finally, in Section \ref{irreducibles} we gather some observations on invariants of blocks whose defect groups are Suzuki $2$-groups using \thref{main_theorem}, known results on irreducible characters of Suzuki $2$-groups and blocks with trivial intersection defect groups.


\section{Suzuki $2$-groups}
\label{Suzuki}

We recall here the description of the classes of Suzuki $2$-groups and give some useful properties.

Write $q=2^m$ for $m \in \NN$. Let $\theta$ be a field automorphism of $\FF_q$ and define $\FF_{\theta} = \{x \in \FF_q:\theta(x)=x\}$. Note $\theta$ is given by $\theta(x)=x^{2^l}$ for some $l$.

Following~\cite{hi63}, the series of Suzuki $2$-groups are as follows. Note that distinct $\theta$ and $\epsilon$ may and do sometimes give isomorphic groups, and that recognising isomorphism is a nontrivial problem.  

\medskip

{\bf Type $\cA$:} When $m$ is not a power of $2$, there exist nontrivial automorphisms $\theta$ of odd order of $\FF_q$. Define $\cA(m,\theta)$ to consist of pairs $(\a,\b)$ where $\a, \b \in \FF_q$, with multiplication given by $$(\a_1,\b_1)(\a_2,\b_2)=(\a_1+\a_2,\b_1+\b_2+\a_1\theta(\a_2)).$$ We have $Z(P)=\{(0,\b):\b \in \FF_q\}$. Note that the Sylow $2$-subgroups of the Suzuki simple group $^2B_2(2^{2t+1})$ are of type $\cA(2t+1,\theta)$ with $\theta(x)=x^{2^{t+1}}$, so that $\theta$ has order $m=2t+1$.

\medskip

{\bf Type $\cB$:} Let $m \geq 2$, $\theta$ be any automorphism of $\FF_q$, and $\epsilon \in \FF_q$ such that there is no $\rho \in \FF_q$ with $\epsilon=\rho^{-1}+\theta(\rho)$. Define $\cB(m,\theta,\epsilon)$ to consist of triples $(\a,\b,\g)$ where $\a,\b,\g \in \FF_q$, with multiplication given by 
$$(\a_1,\b_1,\g_1)(\a_2,\b_2,\g_2)=(\a_1+\a_2,\b_1+\b_2,\g_1+\g_2+\a_1\theta(\a_2)+\epsilon \a_1 \theta(\b_2)+\b_1 \theta(\b_2)).$$ We have $Z(P)=\{(0,0,\g):\g \in \FF_q\}$. Note that the Sylow $2$-subgroups of the groups $PSU_3(2^m)$ are Suzuki groups of type $\cB$.

\medskip

{\bf Type $\cC$:} Let $m \geq 3$ be odd, $\theta$ be the unique automorphism satisfying $2 \theta^2=1$, and $\epsilon \in \FF_q$ such that there is no $\rho \in \FF_q$ with $\epsilon=\rho^{-1}+\rho\theta(\rho^2)$. Define $\cC(m,\theta,\epsilon)$ to consist of triples $(\a,\b,\g)$ where $\a,\b,\g \in \FF_q$, with multiplication given by
$$(\a_1,\b_1,\g_1)(\a_2,\b_2,\g_2)=(\a_1+\a_2,\b_1+\b_2,\g_1+\g_2+\a_1\theta(a_2)+\epsilon \a_1^{1/2} \theta(\b_2^2)+\b_1\b_2).$$ We have $Z(P)=\{(0,0,\g):\g \in \FF_q\}$. 

\medskip

{\bf Type $\cD$:} Let $m \geq 5$ be divisible by $5$, $\theta$ be an automorphism of $\FF_q$ of order $5$, and $\epsilon \in \FF_q$ such that there is no $\rho \in \FF_q$ with $\epsilon=\rho^{-1}+\rho\theta^4(\rho)\theta(\rho)$. Define $\cD(m,\theta,\epsilon)$ to consist of triples $(\a,\b,\g)$ where $\a,\b,\g \in \FF_q$, with multiplication given by
$$(\a_1,\b_1,\g_1)(\a_2,\b_2,\g_2)=(\a_1+\a_2,\b_1+\b_2,\g_1+\g_2+\a_1\theta(a_2)+\epsilon \theta^3(\a_1) \theta(\b_2)+\b_1\theta^2(\b_2)).$$ We have $Z(P)=\{(0,0,\g):\g \in \FF_q\}$.

\bigskip

A search using the SmallGroups library~\cite{gap_SG} tells us the following:

\begin{lemma}
\thlabel{small_cases}
The Suzuki $2$-groups of order $2^6$ are the Sylow $2$-subgroups of $^2B_2(8)$ and $PSU_3(4)$.
\end{lemma}

A feature of Suzuki $2$-groups that helps place restrictions on the structure of groups having a block with these defect groups is the following restriction on the nature of their normal subgroups, mostly based on~\cite{lmm15}.

\begin{lemma}
\thlabel{subgroup_structure}
Let $P$ be a Suzuki $2$-group and $Q \lhd P$.
\begin{enumerate}[(a)]
\item Suppose $P$ is of type $\cA (m,\theta)$, where $\theta$ has (odd) order $k$ and $m=nk$. 
\begin{enumerate}[(i)]
\item If $|QZ(P)/Z(P)| \geq 2$, then $|Z(P) \cap Q| \geq 2^{n(k-1)}$.
\item If $|QZ(P)/Z(P)| \geq 4$, then $Z(P) \leq Q$ and $Z(Q)=Z(P)$.
\item If $|QZ(P)/Z(P)| \leq 2$, then $Q$ is abelian. 
\item If $|Z(P)|=8$ and $|QZ(P)/Z(P)| = 2$, then $Z(P) \leq Q$. 
\end{enumerate}
\item Suppose $P$ has type $\cB$, $\cC$ or $\cD$. If $|QZ(P)/Z(P)| \geq 2$, then $Z(P) \leq Q$ and $Z(Q)=Z(P)$.
\end{enumerate}
\end{lemma}

\begin{proof}
(a) Write $q=2^m$. Let $(\a,\b) \in Q$. For all $x \in \FF_q$, 
$$[(\a,\b),(x,0)]=(0,\a \theta(x)+x \theta(\alpha)) \in Q.$$
Note that $\tau_\a:\FF_q \rightarrow \FF_q$ given by $\tau_\a(x)=\a \theta(x)+x \theta(\alpha)$ is a group homomorphism. By the discussion following~\cite[Remark 2.2]{lmm15}, the subgroup $O_\alpha:=\{(0,\beta):\beta \in \im(\tau_\a)\}=[(\alpha,x),P] \leq Z(P)$ for all $x \in \FF_q$, and has order $2^{n(k-1)}$. Hence if the normal subgroup $Q$ of $P$ possesses an element $(\a,\b)$ outside of $Z(P)$, then $O_\a \leq Q$. The first part in now immediate. 

By~\cite[Proposition 2.4]{lmm15} if $\a_1 \neq \a_2$, both nonzero, then $O_{\a_1}O_{\a_2}=Z(P)$. Hence if $|QZ(P)/Z(P)| \geq 4$, then since there are elements $(\a_1,\b_1), (\a_2,\b_2) \in Q \setminus Z(P)$ with $\a_1 \neq \a_2$, we must have $Z(P)=[Q,P] \leq Q$.

If $|QZ(P)/Z(P)| = 2$, then there is $\alpha \in \FF_q^\times$ such that $Q \leq \langle (\alpha,0) \rangle Z(P)$ and so $Q$ is abelian since $[(\a,\b_1),(\a,\b_2)]=(0,0)$ for all $\b_1,\b_2\in\FF_q$.

Finally suppose that $m=3$ and $|QZ(P)/Z(P)| = 2$. By \thref{small_cases} there is just one possibility for $P$, a Sylow $2$ subgroup of $^2B_2(8)$. The subgroup structure of $P$ may then easily be determined, to show that there are no normal subgroups isomorphic to $C_4$ or $C_4 \times C_2$.

(b) By~\cite[Satz 2]{be77} $P$ is an ultraspecial $2$-group, and hence semi-extraspecial, meaning that for every maximal subgroup $N$ of $Z(P)$, the group $P/N$ is extraspecial. The result then follows by~\cite[Theorem A]{fm01} (or more explicitly~\cite[Corollary 8.3]{fm01}).
\end{proof}

\begin{proposition}
\thlabel{type_A_auts}
Let $P$ be a Suzuki $2$-group of type $\cA$ with $|Z(P)|=q=2^m$. Then $\Aut(P)/O_2(\Aut(P)) \cong C_{q-1} \rtimes C_m$, where a generator of $C_{q-1}$ corresponds to a Singer cycle and $C_m$ corresponds to field automorphisms of $\FF_q$. The automorphisms of $P$ act faithfully on $Z(P)$ and on $P/Z(P)$.
\end{proposition}

\begin{proof}
By~\cite[Theorem 1]{br81} $\Aut(P)$ is solvable. Let $\Lambda:\Aut(P) \rightarrow \Aut(P/\Phi(P)) \cong GL_m(2)$ be the natural map. Then $\ker(\Lambda)$ is a $2$-group (see for example the proof of~\cite[Satz 3.17]{hu67}), so odd order automorphisms of $P$ correspond to odd order automorphisms of $P/\Phi(P)$. By the definition of a Suzuki $2$-group there is $\varphi \in \Aut(P)$ of order $q-1$ permuting the nontrivial elements of $Z(P)$ transitively, and so $\Lambda(\varphi)$ is a Singer cycle. By~\cite{ka80} $\im(\Lambda)$ contains $GL_{m/s}(2^s)$ as a normal subgroup for some $s$. Since $\im(\Gamma)$  is solvable, we must have $s=m$. The normalizer of a Singer subgroup in $GL_m(2)$ has the form $C_{q-1} \rtimes C_m$, hence so does $\im(\Lambda)$, noting that each field automorphism of $\FF_q$ gives rise to an automorphism of $P$ in the obvious way. The result follows, noting that in our situation $O_2(\im(\Lambda))=1$.
\end{proof}



\section{Controlled blocks and normal subgroups of Suzuki $2$-groups}
\label{Controlled}

A $p$-group $P$ is called resistant if every saturated fusion system on $P$ is given $P \rtimes E$ for some $p'$-group $E$. For background on fusion systems see~\cite{ako} or~\cite{craven}, and for more on resistant $p$-groups see~\cite{st06}. As noted in~\cite[Theorem 4.4]{cg12}, Suzuki $2$-groups are resistant since $Z(P)$ consists of the identity element and all involutions in $P$. Further the same is true for all normal subgroups of Suzuki $2$-groups once we note that abelian $p$-groups are resistant:

\begin{proposition}
\thlabel{Suzuki_A_resistant}
Let $P$ be a Suzuki $2$-group and let $Q \lhd P$. Then $Q$ is resistant.
\end{proposition}

\begin{proof}
By \thref{subgroup_structure} either $Q$ is abelian or $Z(Q)$ consists precisely of the identity element and the involutions in $Q$. In either case, it follows from~\cite[Theorem 4.8]{st06} that $Q$ is resistant.
\end{proof}

Before proceeding we recall definitions of subpairs and the inertial subgroup.

A $B$-subpair is a block $B$ of a group $G$ is a pair $(Q,b_Q)$, where $Q$ is a $p$-subgroup of $G$ and $b_Q$ is a
block of $QC_G(Q)$ with Brauer correspondent $(b_Q)^G=B$. The
$B$-subpairs with $|Q|$ maximized are called the Sylow
$B$-subpairs, and they are the $B$-subpairs for which $Q$ is a defect group 
of $B$. Letting $P$ be a defect group of $B$, we denote by $N_G(P,b_P)$ the stabilizer in $N_G(P)$ of $(P,b_P)$ under conjugation.

The \emph{inertial quotient} of $B$ is $E=N_G(P,b_P)/PC_G(P)$, together with the action of $E$ on $P$, and is determined by the fusion system $\mathcal{F}=\mathcal{F}_{(P,b_P)}(G,B)$ for $B$, sometimes called the Frobenius category. We refer to~\cite[Section 8.5]{lin2} for background on this. Note that $E$ is a $p'$-group. Basic Morita equivalence of blocks of finite groups preserves the Frobenius category (see~\cite[Section 9.10]{lin2}).

Following the presentation in~\cite[Section 8.14]{lin2}, a K\"ulshammer-Puig class is an element of $H^2(\Aut_\mathcal{F}(P),k^\times)$, which is isomorphic to $H^2(E,k^\times)$ (see~\cite[Remark 8.14.3]{lin2}). 

By~\cite[Theorem 6.14.1]{lin2} a block with normal defect group is determined up to basic Morita equivalence by the inertial quotient and K\"ulshammer-Puig class. Since $H^2(L,k^\times)$ is trivial when $L$ is cyclic (see for example~\cite[Proposition 1.2.10]{lin1}), it follows from~\cite[Proposition 1.2.15]{lin1} that if $E$ has cyclic Sylow $l$-subgroups for all primes $l$, then $H^2(E,k^\times)$ is trivial. 

The block $B$ is controlled if the fusion system $\mathcal{F}_{(P,b_P)}(G,B)$ is the same as that of its Brauer correspondent in $N_G(P)$. Every block with resistant defect group is controlled, so we have: 

\begin{corollary}
\thlabel{Suzuki_controlled}
Let $G$ be a finite group and $B$ be a block of $\cO G$ with defect group $P$ which is either a Suzuki $2$-group or a normal subgroup of a Suzuki $2$-group. Then $B$ is a controlled block.
\end{corollary}

Recall that $B$ is nilpotent if $\mathcal{F}_{(P,b_P)}(G,B)=\mathcal{F}_P(P)$. A controlled block is nilpotent precisely when the inertial quotient is trivial, a fact we will use frequently and without reference throughout.

\begin{corollary}
\thlabel{Z(P)central}
Let $G$ be a finite group and $B$ be a block of $\cO G$ with defect group $P$ which is a Suzuki $2$-group of type $\mathcal{A}$. If $Z(P) \leq Z(G)$, then $B$ is nilpotent.
\end{corollary}

\begin{proof}
Since $B$ is controlled, it suffices to show that $N_G(P)/C_G(P)$ is a $2$-group. By \thref{type_A_auts} odd order automorphisms of $P$ are described by Singer cycles and field automorphisms, neither of which fix all elements of $Z(P)$, so we are done.
\end{proof}

\begin{remark}
\thlabel{nilpotency_criterion}
A block with normal defect group whose inertial quotient has cyclic Sylow $l$-subgroups for all primes $l$ is nilpotent if and only if the number $l(B)$ of simple $B \otimes_\cO k$-modules is $1$.
\end{remark}

Controlled $2$-blocks of quasisimple groups have been described by An in~\cite[Theorem 1.1]{an20}, from which we have the following, noting that the Sylow $2$-subgroups of $PSU_3(2^n)$ and $^2B_2(2^{2n+1})$ are indeed Suzuki $2$-groups:

\begin{proposition}
\thlabel{Suzuki_quasisimple}
Let $G$ be a quasisimple group and $B$ a block of $\cO G$ with defect group $P$ which is a Suzuki $2$-group or a nonabelian normal subgroup of a Suzuki $2$-group. Then $P \in \Syl_2(G)$ and $G/Z(G)$ is $PSU_3(2^n)$ or $^2B_2(2^{2n+1})$ for some $n$.
\end{proposition}


\section{Reductions and proof of \thref{main_theorem}}
\label{Proofs}

We show that every block whose defect groups are Suzuki $2$-groups is basic Morita equivalent to what we call a reduced block. We will then show that finite groups with such a reduced block have a very restricted structure.

The following result is used in previous reductions for results concerning Morita equivalence classes of blocks, and encapsulates the use of Fong-Reynolds reductions and the K\"ulshammer-Puig reductions~\cite{kp90}. Recall that a block $B$ is \emph{quasiprimitive} if every block of every normal subgroup covered by $B$ is $G$-stable. In particular $B$ covers a unique block for each normal subgroup. 

\begin{lemma}[Proposition 6.1 of~\cite{ae23}]
\thlabel{reduced_block}
Let $G$ be a finite group and $B$ a block of $\cO G$ with defect group $P$. Then there is a finite group $H$ and a block $C$ of $\cO H$ such that $B$ is basic Morita equivalent to $C$, a defect group $P_H$ of $C$ is isomorphic to $P$ and:
\begin{enumerate} 
\item[(R1)] $C$ is quasiprimitive;
\item[(R2)] If $N \lhd H$ and $C$ covers a nilpotent block of $\cO N$, then $N \leq O_p(H)Z(H)$ with $O_{p'}(N) \leq [H,H]$ cyclic. In particular $O_{p'}(H) \leq Z(H)$.
\end{enumerate}
Note that $B$ and $C$ have the same Frobenius category $\mathcal{F}$, and the same K\"ulshammer-Puig class in $H^2(\Aut_\mathcal{F}(P),k^\times)$.
\end{lemma}

We call the pair $(H,C)$, where $C$ is a block of $\cO H$, \emph{reduced} if it satisfies conditions (R1) and (R2) of \thref{reduced_block}. If the group is clear, then we just say $C$ is reduced.

Before proceeding we recall some definitions concerning the generalized Fitting subgroup. Details may be found in~\cite{asc00}. A component of $G$ is a subnormal quasisimple subgroup. The components of $G$ commute, and we define the layer $E(G)$ of $G$ to be the normal subgroup of $G$ generated by the components. The layer is a central product of the components. The Fitting subgroup $F(G)$ is the largest nilpotent normal subgroup of $G$, and is the direct product of $O_l(G)$ for all primes $l$ dividing $|G|$. The \emph{generalized Fitting subgroup} is $F^*(G)=E(G)F(G)$. This has the property that $C_G(F^*(G)) \leq F^*(G)$. 

Let $B$ be a quasiprimitive block of $G$ with defect group $P$. Then $B$ covers a unique block $B_E$ of $E(G)$, and this has defect group $P \cap E(G)$ (see~\cite[Theorem 15.1]{alp86}). Let $S$ be a component of $G$. Then since $E(G)$ is a central product of the components, $B_E$ covers a unique block $B_S$ of $S$, and this has defect group $P \cap S$.

\begin{lemma}
\thlabel{component_not_nilpotent}
With the notation above, if $B_S$ is nilpotent, then $B_E$ is nilpotent.
\end{lemma}

\begin{proof}
This argument may be found in the proof of~\cite[Proposition 4.3]{eel20}.    
\end{proof}

\begin{lemma}
\thlabel{rank_out_mult}
\begin{enumerate}[(i)]
    \item The $2$-rank of the centre of a quasisimple group is at most two.
    \item The $2$-rank of any section of the outer automorphism group $\Out(S)$ of a nonabelian simple group $S$ is at most three. Further, if $H \leq \Out(S)$ has $2$-rank $3$, then there is $N \leq H$ with $H/N \cong C_2 \times C_2$.
\end{enumerate}
\end{lemma}

\begin{proof}
This may be checked in~\cite{atlas}, using~\cite[Theorem 2.5.12]{gls} for a detailed description of $\Out(S)$ where necessary.
\end{proof}

\begin{lemma}
\thlabel{Sz(8)_even_index}
If $C$ is a quasiprimitive non-nilpotent block of a finite group $G$ with defect group $P$ of type $\cA$ with $|Z(P)|=8$, then $G$ cannot have a normal subgroup $H$ with $G/H \cong C_2 \times C_2$.
\end{lemma}

\begin{proof}
Note that we have $|\Out(P)|_{2'}=21$ and $\Out(P)$ contains a subgroup $C_7 \rtimes C_3$.

Consider the Brauer correspondent $b$ of $C$ in $N:=N_{G}(P)$. In order to use an argument normally applied to blocks with abelian defect groups, we work with $\overline{N}:=N_{G}(P)/Z(P)$. By~\cite[Corollary 4]{ku87} there is a unique block $\overline{b}$ of $\overline{N}$ dominated by $b$ (see~\cite{nt} concerning domination of blocks). By~\cite[Theorem 5.8.10]{nt} $\overline{b}$ must necessarily have defect group $P/Z(P)$. Let $(\overline{P},\overline{b}_{\overline{P}})$ be a $\overline{b}$-subpair, so that $\overline{b}$ has inertial quotient $E:=N_{\overline{N}}(\overline{P},\overline{b}_{\overline{P}})/C_{\overline{N}}(\overline{P})$. Now we cannot have $E=1$, for otherwise $\overline{b}$ would be nilpotent, in which case $l(b)=l(\overline{b})=1$ and $b$ would have trivial inertial quotient by \thref{nilpotency_criterion} (noting that the inertial quotient of $b$ is a subgroup of $C_7 \rtimes C_3$, and so has trivial Schur multiplier), which would in turn imply that $B$ would have to be nilpotent since $B$ is controlled with inertial quotient a subgroup of $C_7 \rtimes C_3$. It follows from~\cite[Theorem 5.2.3]{gor} that $\overline{P}=[\overline{P},N_{\overline{N}}(\overline{P},\overline{b}_{\overline{P}})] \times C_{\overline{P}}(N_{\overline{N}}(\overline{P},\overline{b}_{\overline{P}}))$. Suppose that $H \lhd G$ with $G/H \cong C_2 \times C_2$. Then $N=P(H \cap N)$ and $\overline{N}=\overline{P}(\overline{H} \cap \overline{N})$, and we note that $[\overline{P},N_{\overline{N}}(\overline{P},\overline{b}_{\overline{P}})] = [\overline{P},\overline{P}N_{\overline{N} \cap \overline{H}}(\overline{P},\overline{b}_{\overline{P}})] \leq \overline{N} \cap \overline{H}$. Since $|\overline{N} \cap \overline{H}|=2$, we have $C_{\overline{P}}(N_{\overline{N}}(\overline{P},\overline{b}_{\overline{P}})) \cong C_2 \times C_2$, so that $E$ acts trivially on a subgroup of $\overline{P}$ of order $4$. Hence $E$ must act trivially on $\overline{P}$, and so $E=1$, which we have already established cannot happen.
\end{proof}

\begin{proposition}
\thlabel{main_reduction}
Let $(G,B)$ be a reduced pair where $P$ is a Suzuki $2$-group. Then one of the following:

\begin{enumerate}[(i)]
\item $P \lhd G$;
\item $^2B_2(2^{2n+1}) \leq G \leq \Aut({}^2B_2(2^{2n+1}))$ for some $n \geq 1$ and $B$ is the principal block;
\item $Z(G) \leq [G,G]$ and $PSU_3(2^n) \leq G/Z(G) \leq \Aut(PSU_3(2^n))$ for some $n \geq 2$, and $B$ is a block of maximal defect. 
\end{enumerate}
\end{proposition}

We prove \thref{main_reduction} in a series of steps.

Let $(G,B)$ be reduced, and suppose that $O_2(G) \neq P$, so $O_2(G)$ is a proper subgroup of $P$.

\begin{step}
\thlabel{has_component}
$G$ has at least one component.
\end{step}

\begin{proof}
Suppose that $E(G) = 1$, so that $F^*(G)=F(G)=O_2(G)Z(G)$. Then $C_G(O_2(G))=C_G(F^*(G)) \leq F^*(G)=O_2(G)Z(G)$. In particular $Z(P) < O_2(G)$, so that $\Omega_1(O_2(G)) = Z(P)$ and $P \leq C_G(\Omega_1(O_2(G)))=:N \lhd G$.  

Let $B_N$ be the (unique) block of $\cO N$ covered by $B$, and note that it has $P$ as a defect group. Write $Z=\Omega_1(O_2(G))$ and $\overline{G}=G/Z$. Note that $O_{2'}(\overline{N}) = O_{2'}(G)Z/Z$, so in particular $O_{2'}(\overline{N}) \leq Z(\overline{N})$. Let $B_{\overline{N}}$ be the unique block of $\overline{N}$ corresponding to $B_N$, so that $B_{\overline{N}}$ has defect group $\overline{P}$, an abelian group. Suppose that $E(\overline{N})=1$. Then 
$F^*(\overline{N})=O_2(\overline{N})O_{2'}(\overline{N})$ and $\overline{P} \leq C_{\overline{N}}(O_2(\overline{N})) \leq O_2(\overline{N})O_{2'}(\overline{N})$. Hence $P=O_2(N)$ and $P \lhd G$, contradicting our assumption. So $E(\overline{N}) \neq 1$. Let $Y$ be a component in $E(\overline{N})$ and let $X$ be the preimage of $Y$ in $N$, so that $X$ is a central product of a component $S$ of $N$ with $Z$. Hence $S$ is a component of $G$, a contradiction. 
\end{proof}

\begin{step}
\thlabel{components}
Write $1 \neq E(G)=L_1 \cdots L_t \lhd G$, where the $L_i$ are the components of $G$. Write $B_E$ for the unique block of $E(G)$ covered by $B$, and note that this has defect group $P \cap E(G)$. Since each $L_i$ is normal in $E(G)$, we may choose a block $B_i$ of $L_i$ covered by $B_E$ with defect group $P \cap L_i$. Then no $B_i$ is nilpotent. In particular, for each $i$ we have $O_2(L_i) < P \cap L_i$.
\end{step}

\begin{proof}
See \thref{component_not_nilpotent}, noting that blocks with central defect group are nilpotent.
\end{proof}

\begin{step}
\thlabel{E(G)inT}
$E(G) \leq T:= \langle P^g:g \in G \rangle$.
\end{step}

\begin{proof}
It follows from \thref{components} that $E(G)=\langle (P \cap E(G))^g : g \in G \rangle \leq T$.
\end{proof}

\begin{step}
\thlabel{type_AB_unique_comp}
If $P$ is of type $\cA$ with $|Z(P)|=8$ or type $\cB$ with $|Z(P)|=4$, then there is a unique component.
\end{step}

\begin{proof}
Note that $P \cap E(G) \lhd P$, and that $P \cap E(G)$ is a central product of the $P \cap L_i$. Suppose $P$ is of type $\cA$. If $t \geq 2$, then $P \cap L_i$ must be cyclic for some $i$, in which case $B_i$ is nilpotent, contradicting \thref{components}. Hence $t=1$ in this case. Similarly if $P$ is of type $\cB$, then we have $t \leq 2$, and if $t=2$, then $P \leq L_1L_2$ with $(P \cap L_i)Z(P)/Z(P) \cong C_2 \times C_2$ for each $i$. By \thref{small_cases} there is only one possibility for $P$, and it is easily seen that this cannot be written as a central product of two such groups.
\end{proof}

\begin{step} 
\thlabel{comp_defect_no_in_Z(P)}
If $L_i$ is a component with $P \cap L_i \leq Z(P)$, then $L_i$ is the unique component and $[P:Z(P)]=16$ (in which case $|Z(P)|=4$) or $[P:Z(P)]=8$ (in which case $|Z(P)|=8$).
\end{step}

\begin{proof}
Consider $L_i$ with $P \cap L_i \leq Z(P)$. Then $P \leq N_G(L_{i})$ for each $i$, and so $L_{i} \lhd T$. 

Consider $C_T(L_i) \lhd T$, and let $C$ be a block of $C_T(L_i)$ covered by the unique block $B_T$ of $T$ covered by $B$. Note that $B_T$ has defect group $P$, and we may choose $C$ with defect group $C_P(L_i)$. Suppose that $|C_P(L_i)Z(P)/Z(P)| > 2$. Then by \thref{subgroup_structure} $Z(P) \leq C_P(L_i)$, so that $P \cap L_i = Z(P) \cap L_i \leq Z(L_i)$, contradicting \thref{components}. Hence $|C_P(L_i)Z(P)/Z(P)| \leq 2$. Suppose that $[P:Z(P)] \geq 2^5$. Then we have $P/C_P(L_i)Z(P) \geq 2^4$, so that $P/C_P(L_i)Z(P)$ has $2$-rank at least $4$. But $P/C_P(L_i)Z(P)$ is isomorphic to a section of $\Out(L_i/Z(L_i))$, which is impossible by \thref{rank_out_mult}. We are left with $[P:Z(P)] \leq 16$. This means $P$ is either of type $\cA$ with $|Z(P)|=[P:Z(P)]=8$ or $P$ is of type $\cB$ with $[P:Z(P)]=16$ and $|Z(P)|=4$. In either case $L_i$ must be the unique component by \thref{type_AB_unique_comp}.
\end{proof}

\begin{step} 
\thlabel{components_normal}
$P \cap L_i \lhd P$ for each $i$.
\end{step}

\begin{proof}
By \thref{comp_defect_no_in_Z(P)} we may assume $(P \cap L_i)Z(P) > Z(P)$.

Let $L_{i_1}, \ldots , L_{i_r}$ be the $G$-orbit containing $L_i$, and suppose that $r \geq 2$. For the moment write $L=L_{i_1} \cdots L_{i_r}$. Then $|(P \cap L)Z(P)/Z(P)|>2$, so by \thref{subgroup_structure} $Z(P) \leq L$ and $Z(P \cap L) = Z(P)$.

Suppose first that $Z(P) \leq Z(L)$, so that (since $Z(P \cap L) = Z(P)$ and $O_2(Z(L)) \leq P$) we have $Z(P)=O_2(Z(L))$. Note that $|(P \cap L_{i_j})Z(L)/Z(L)| > 1$ for each $i_j$. Since $P/Z(P)$, and so $PZ(L)/Z(L)$, is abelian, it follows that $PZ(L)/Z(L)$ normalizes each $L_{i_j}Z(L)/Z(L)$, as $L/Z(L)$ is the direct product of the nonabelian simple groups $L_{i_j}Z(L)/Z(L)$. Hence $P$ normalizes each $P \cap L_{i_j}$. Suppose that $Z(P) \not\leq Z(L_{i_1} \cdots L_{i_r})$. Then each $P \cap L_{i_j}$ contains an involution in $P$ non-central in $L_{i_j}$. It follows that $P$ normalizes each $L_{i_j}$, and so normalizes $P \cap L_{i_j}$.
\end{proof}

\begin{step}
\thlabel{one_component}
$G$ has only one component, i.e., $t=1$.
\end{step}

\begin{proof}
By \thref{comp_defect_no_in_Z(P)} for each $i$ we have $(P \cap L_i)Z(P) > Z(P)$ and by \thref{components_normal} $P$ normalizes each $P \cap L_i$. Suppose that $P$ is not of type $\cA$. Then by \thref{subgroup_structure} $Z(P)$ is contained in a defect group of $B_i$ for each $i$. Hence if $t \geq 2$, then $Z(P) \leq Z(L_i)$ for each $i$ (since each $L_i$ contains all of the involutions in $P$). By \thref{rank_out_mult}, the $2$-rank of $Z(L_i)$ is at most two, so that $|Z(P)|=4$ and $|P/Z(P)|=16$. But then $L_i$ is the unique component by \thref{type_AB_unique_comp}.

Suppose that $P$ is of type $\cA$ and that $t \geq 2$. Write $[P:Z(P)]=|Z(P)|=2^m$. By \thref{subgroup_structure} for all $i \neq j$ we have $Z(P) \leq L_iL_j$. If $t > 2$, then this means that $Z(P) \leq Z(L_i)$ for each $i$, which is impossible by \thref{rank_out_mult} since $m \geq 3$. Hence $t=2$. Note that we have $Z(P) \leq E(G)$.

Suppose that $O_2(G) \not\leq E(G)$. Then possesses an element of order four outside of $E(G)$. So $(P \cap L_i)O_2(G) \lhd P$ contains two distinct elements of order $4$ and by \thref{subgroup_structure} $Z(P) \leq (P \cap L_i)O_2(G)$ for each $i$. But $Z(P) \leq E(G)$, so $Z(P) \leq L_i$ for each $i$, and it follows that $Z(P) \leq Z(E(G))$. Hence $Z(P) \lhd G$ and we may consider $C_G(Z(P)) \lhd G$. By \thref{Z(P)central} the (unique) block of $C_G(Z(P))$ covered by $B$ is nilpotent, a contradiction. Hence $O_2(G) \leq E(G)$.

If $[(P \cap L_i)Z(P):Z(P)] >2$ for some $i$, then by \thref{subgroup_structure} $Z(P) \leq L_i$, and as above we must have $Z(P) \leq Z(L_j)$ for each $j$, which again is impossible as $m \geq 3$ and $Z(L_j)$ has $2$-rank at most $2$. Hence $P \cap L_i \cong C_4 \times (C_2)^{s_i}$ for some $s_i$.

Taking stock, we have $F^*(G)/Z(G)=F^*(G/Z(G)) \cong L_1/Z(L_1) \times L_2/Z(L_2)$. Since $C_{G/Z(G)}(F^*(G/Z(G))) \leq F^*(G/Z(G))$ we have that $G/F^*(G)$ embeds in $\Out(F^*(G/Z(G)))$, which is isomorphic to a subgroup of $\Out(L_1/Z(L_1)) \wr C_2$ or $\Out(L_1/Z(L_1)) \times \Out(L_2/Z(L_2))$, depending on whether or not $L_1 \cong L_2$. By~\cite[Theorem 15.1]{alp86}) our assumption that $B$ is reduced implies that $PN=G$ for any subgroup $N$ of $G$ of index $2$. But \thref{components_normal} implies that no element of $P$ can switch the two components, so each component is normal in $G$ and $G/F^*(G)$ embeds in $\Out(L_1/Z(L_1)) \times \Out(L_2/Z(L_2))$. Counting elements of order $4$, since $\Out(L_i)$ has $2$-rank at most $3$, this means $m \leq 8$. By definition we cannot have $m=8$. Since blocks with defect group $C_4 \times C_2$ are nilpotent, we must have $|(Z(P) \cap L_i)/O_2(L_i)| \geq 2^3$ (we are using the fact that by~\cite{wa94} a $p$-block $b$ of a finite group $X$ is nilpotent if and only if the corresponding block of $X/O_p(Z(X))$ is nilpotent). Hence $6 \leq m \leq 7$.

If $m=6$, then we must have $|(Z(P) \cap L_i)/O_2(L_i)| = |Z(P) \cap L_i| = 2^3$ for each $i$. But in the notation of \thref{subgroup_structure} $k=3$ and $n=2$, so that $|L_i \cap Z(P)| \geq 2^{(k-1)n}=2^4$, a contradiction. If $m=7$, then $k=7$ and $n=1$, so that by \thref{subgroup_structure} we have $|L_i \cap Z(P)| \geq 2^{(k-1)n}=2^6$ for each $i$, a contradiction. 

Hence $t=1$ and we are done.
\end{proof}

\begin{step}
\thlabel{Sz(8)}
If $P$ is of type $\cA$ with $|Z(P)|=8$, then $P \leq {}^2B_2(2^{3}) \leq G \leq \Aut({}^2B_2(2^{3}))$.
\end{step}

\begin{proof}
$G$ has a unique component $L_1$, and the block $B_1$ of $L_1$ covered by $B$ has defect group $P_1:=P \cap L_1$. 

If $|P_1Z(P)/Z(P)| > 2$, then by \thref{subgroup_structure} we have $Z(P) \leq P_1$. By \thref{Suzuki_quasisimple} $L_1 \cong {}^2B_2(2^{8})$ and $P_1=P$. It follows that $O_2(G)=1$ and $F^*(G)=L_1$, so since $C_G(F^*(G)) \leq F^*(G)$ the result follows in this case.

Suppose that $P_1 \leq Z(P)$. Then $P_1 \cong C_2 \times C_2$ or $C_2 \times C_2 \times C_2$ (it cannot be $C_2$ as $B_1$ is not nilpotent). Suppose first that $P_1 \cong C_2 \times C_2$. Then $O_2(L_1)=1$ since otherwise $B_1$ would be nilpotent. Hence $F^*(G)=(L_1 \times O_2(G))O_{2'}(G)$. We have $F^*(G/O_{2'}(G)) \cong L_1/Z(L_1) \times O_2(G)$, where $O_2(G)$ is cyclic, and so $G/F^*(G)$ is isomorphic to a subgroup of $\Out(O_2(G) \times L_1/Z(L_1)) \cong \Out(L_1/Z(L_1))$. If $O_2(G) \cong C_2$ or is trivial, then $G/L_1$ has $2$-rank at least three, and so by \thref{rank_out_mult} $G/L_1$ has a normal subgroup with quotient $C_2 \times C_2$, which cannot happen by \thref{Sz(8)_even_index}. If $O_2(G) \cong C_4$, then $G$, and so $P$, has $O_2(G)$ as a direct factor. But $P$ cannot be factorised in this way, a contradiction. If $P_1 \cong C_2 \times C_2 \times C_2$. Then $O_2(G) \leq L_1$, and again $G/L_1$ has $2$-rank three, so that $G/L_1$ has a normal subgroup with quotient $C_2 \times C_2$, which cannot happen by \thref{Sz(8)_even_index}.

Finally, suppose that $|P_1Z(P)/Z(P)| = 2$. Now $P$ has no normal subgroup isomorphic to $C_4$ or $C_4 \times C_2$, so $P_1 \cong C_4 \times C_2 \times C_2$. Further $O_2(G)=1$. We have $P/P_1 \cong C_2 \times C_2$ and $F^*(G)=L_1Z(G)$. Since $G/F^*(G)$ is isomorphic to a subgroup of $\Out(L_1/Z(L_1))$, it is solvable (since the outer automorphism group of a nonabelian simple group is always solvable), and so is $G/L_1$. By~\cite[Lemma 2.4]{ar21} $PL_1/L_1$ is a Sylow $2$-subgroup of $G/L_1$. Since $PL_1/L_1$ is an abelian, $G$ has normal subgroups $G_1$, $G_2$ such that $|G/G_1|$ is odd and $G_2 \leq G_1$ with $G_1/G_2 \cong C_2 \times C_2$. Now $B$ covers a unique block $C$ of $G_1$, and this also has defect group $P$. But $C$ is not nilpotent, so by \thref{Sz(8)_even_index} this configuration cannot arise.
\end{proof}

\begin{step}
\thref{main_reduction} holds.
\end{step}

\begin{proof}
By the previous steps, if $P \neq O_2(G)$, then we have $F^*(G) = L_1O_2(G)Z(G)$. Recall that $B_1$ is the block of $L_1$ covered by $B$ and that $P_1 : = P \cap L_1$ is a defect group of $B_1$. 

\emph{Suppose that $P_1Z(P)>Z(P)$ and that $P$ is not of type $\cA$}. Then by \thref{subgroup_structure} $Z(P) < P_1$, so $P_1$ is a nonabelian normal subgroup of $P$, and by \thref{Suzuki_quasisimple} $L_1/Z(L_1) \cong PSU_3(2^n)$ for some $m$. In this case $O_2(L_1)=1$. Since $P_1$ contains all involutions in $P$, this means that $O_2(G)=1$ and $F^*(G)=L_1Z(G)$. Since $C_G(F^*(G)) \leq F^*(G)$, the result follows in this case. 

\emph{Suppose that $[P_1Z(P):Z(P)]>2$ and that $P$ is of type $\cA$}. Then by \thref{subgroup_structure} $Z(P) \leq P_1$ with $[P:Z(P)]>2$, so $P_1$ is a nonabelian normal subgroup of $P$, and by \thref{Suzuki_quasisimple} $L_1 \cong {}^2B_2(2^{2n+1})$ for some $n$. Since $P_1$ contains all involutions in $P$, this means that $O_2(G)=1$ and $F^*(G)=L_1$, noting that $Z(G)=1$ as $\Aut(L_1)$ has trivial Schur multiplier in this case. Since $C_G(F^*(G)) \leq F^*(G)$, the result follows in this case. 

\emph{Suppose that $[P_1Z(P):Z(P)]=2$ and that $P$ is of type $\cA$}. Write $|Z(P)|=[P:Z(P)]=2^m$. In this case $P_1$ is $C_4 \times (C_2)^s$ for some $s$. 

Write $H=C_G(\Omega_1(O_2(G))) \leq G$. Then $P \leq H$ and $B$ covers a unique block $B_H$ of $H$, and this has defect group $P$. We have $|O_2(G)Z(P)/Z(P)| \leq 2$, for otherwise by \thref{subgroup_structure} $Z(P) \leq O_2(G)$ and $Z(P) \leq Z(H)$, which would then imply that $B_H$ is nilpotent by \thref{Z(P)central}. Say $O_2(G) \cong (C_4)^i \times (C_2)^u$ for some $u,i$, where $i \in \{ 0,1 \}$.

Note that $L_1 \leq H$, that $O_2(H)=O_2(G)$, and that $O_{2'}(G)=O_{2'}(Z(G))=O_{2'}(H) \leq Z(H)$, so that $F^*(H)=F^*(G)$. Hence we have that $H/F^*(H)$ embeds in $\Out(L_1)$. By \thref{rank_out_mult} $\Out(L_1)$ has $2$-rank at most three, so $m \leq 5$, with equality only when $i=1$. By \thref{Sz(8)} we may assume $m \neq 3$. Since $m$ cannot be $4$, this leaves $m=5$. Since in this case $i=1$, we have $|P_1O_2(G)Z(P)/Z(P)| >2$. Hence by \thref{subgroup_structure} we have $Z(P) \leq P_1O_2(G) \leq H$. Also note that $O_2(G)$ is not contained in $L_1$. By \thref{type_A_auts} the inertial quotient of $B_H$ is a subgroup of $C_{31} \rtimes C_5$. Elements of order $31$ (corresponding to Singer cycles on $\FF_{2^5}$) act transitively on $Z(P)$ and on $P/Z(P)$, hence the inertial quotient of $B_H$ cannot contain an element of order $31$. Elements of order $5$ act as field automorphisms of $\FF_{2^5}$ on $Z(P)$ and on $P/Z(P)$. Since $5$ is prime (so that the fixed point space of a nontrivial field automorphism has order $2$), the inertial quotient cannot normalize both $O_2(G)$ and $P_1$ and so cannot contain an element of order $5$. Hence the inertial quotient of $B_H$ is trivial and $B_H$ is nilpotent. Hence we have ruled out this configuration.

\emph{Suppose that $P_1 \leq Z(P)$}. By \thref{comp_defect_no_in_Z(P)} and \thref{Sz(8)} we must have that $|Z(P)|=4$ and $[P:Z(P)]=16$, so that $P$ is of type $\cB$. But we must have $O_2(G) \leq L_1$ and $F^*(G)=L_1$, so $G/L_1$ has $2$-rank four, contradicting \thref{rank_out_mult}. Hence this case cannot occur.
\end{proof}

\bigskip

\textbf{Proof of \thref{main_theorem}}
By \thref{reduced_block} and \thref{main_reduction} it suffices to consider blocks with $P \lhd G$. In this case the result follows from the main result of~\cite{ku85} (see also~\cite[Theorem 6.14.1]{lin2}).\hfill $\Box$

\bigskip

As an illustration of \thref{main_theorem}, recalling \thref{small_cases}, we list the Morita equivalence classes for the Suzuki groups of order $64$. Here $B_0(-)$ denotes the principal block.

\begin{corollary}
Let $B$ be a block of a finite group $G$ with defect group $P$. If $P$ is a Suzuki $2$-group of order $64$, then $P$ is of type $\cA$ or $\cB$.
\begin{enumerate}[(a)]
\item If $P$ is a Suzuki group of type $\cA$, then $P$ is isomorphic to a Sylow $2$-subgroup of ${}^2B_2(8)$ and $B$ is basic Morita equivalent to one of the following:
\begin{enumerate}[(i)]
\item $\cO P$
\item $\cO (P \rtimes C_7)$
\item $\cO (P \rtimes C_3)$
\item $\cO (P \rtimes (C_7 \rtimes C_3))$
\item $B_0(\cO ({}^2B_2(8)))$
\item $B_0(\cO (\Aut({}^2B_2(8))))$
\end{enumerate}
\item If $P$ is a Suzuki group of type $\cB$, then $P$ is isomorphic to a Sylow $2$-subgroup of $PSU_3(4)$ and $B$ is basic Morita equivalent to one of the following:
\begin{enumerate}[(i)]
\item $\cO P$
\item $\cO (P \rtimes C_3)$
\item $B_0(\cO PSU_3(4))$
\end{enumerate}
\end{enumerate}
\end{corollary}


\section{Trivial intersection of Suzuki Sylow $2$-subgroups}
\label{TI}

A subgroup $P$ of a finite group $G$ is trivial intersection (TI) if for all $g \in G \setminus N_G(P)$ we have $P^g \cap P = 1$. Inspired by~\cite{bl85} we deduce the following from \thref{main_reduction}.

\begin{corollary}
\thlabel{TIlabel}
Let $G$ be a finite group with Sylow $2$-subgroup $P$ that is a Suzuki $2$-group and suppose that $O_{2'}(G)=1$. Then $P$ is trivial intersection in $G$.
\end{corollary}

\begin{proof}
Consider the principal block of $G$. Since the principal block of a group $H$ is nilpotent precisely when $H$ is $p$-nilpotent (that is, $[H:O_{p'}(H)]=|H|_p$), and since $O_{2'}(G)=1$, we have that $G$ satisfies (R2) of \thref{reduced_block}. Condition (R1) is automatically satisfied for principal blocks, so the principal block of $G$ is reduced. The result then follows from \thref{main_reduction} since all blocks listed there are TI.
\end{proof}


\section{Irreducible characters}
\label{irreducibles}

Morita equivalence of $\cO$-blocks of finite groups preserves numbers of irreducible characters of each height, hence \thref{main_theorem} may be used prove facts about these quantities.

Write $\Irr(B)$ for the set of irreducible characters of $G$ belonging to $B$. Writing $P$ for a defect group of $B$ and $|P|=p^d$, the height $h=h(\chi)$ of $\chi \in \Irr(B)$ is the non-negative integer $h$ such that $p^h[G:P]_p=\chi(1)_p$. Write $\Irr_h(B)$ for the set of irreducible characters in $B$ of height $h$ and $k_h(B)=|\Irr_h(B)|$.

\begin{proposition}
\thlabel{number_irred}
Let $B$ be a block of a finite group $G$ with defect group $P$ that is a Suzuki $2$-group. Write $|Z(P)|=q=2^m$ and let $b$ be the Brauer correspondent block of $B$ in $N_G(P)$. 
\begin{enumerate}[(a)]
\item $k_h(B)=k_h(b)$ for each $h$.
\item If $P$ is not of type $\cA$, then $k_h(B)\neq 0$ precisely when $h=0$ or $h=m$.
\item Suppose $P$ is of type $\cA$, defined by a field automorphism $\theta$ of $\FF_q$. Let $n$ be the order $\theta$ and $r=m/n$.
\begin{enumerate}[(i)]
\item If $n$ is odd, then $k_h(B)\neq 0$ precisely when $h=0$ or $h=(m-r)/2$.
\item If $n=2$, then $k_h(B)\neq 0$ precisely when $h=0$ or $h=m/2$.
\item If $n>2$ is even, then $k_h(B)\neq 0$ precisely when $h=0$, $h=(m-2r)/2$ or $h=m/2$.
\end{enumerate}
\end{enumerate}
\end{proposition}

\begin{proof}
By Theorem \ref{main_theorem} $B$ is Morita equivalent to one of the blocks listed, with the same inertial quotient and K\"ulshammer-Puig class. The same is true of the Brauer correspondent block $b$ of $N_G(P)$. Therefore it suffices to check the result for the listed blocks. Part (a) then follows from~\cite{ea01}.

The conclusion of (b) holds for blocks with normal defect group by~\cite{fm01}, and (c) for blocks with normal defect group by~\cite{bj08} (noting that the result was first proved in~\cite{sa99}). Parts (b) and (c) then follow by (a).
\end{proof}

Brauer's $k(B)$ conjecture states that $k(B) \leq |P|$. Since it is known for $p$-solvable groups, and hence for blocks with normal defect groups (see~\cite{gl84} for the case $p=2$), a consequence of \thref{number_irred} is that the $k(B)$ conjecture holds for blocks whose defect groups are Suzuki $2$-groups. Another consequence of \thref{number_irred} is that the conjecture in~\cite{em14} also holds for blocks with these defect groups.

\begin{center} {\bf Acknowledgments}
\end{center}
\smallskip

We thank Alexander Moret\'o for drawing our attention to~\cite{be77} and~\cite{fm01}, and Shigeo Koshitani for some insightful questions. We also thank the referee for their careful reading of the manuscript, which greatly helped the presentation.

\end{document}